\documentclass{amsart}

\usepackage{amsmath,amssymb}
\usepackage{amsthm}
\usepackage{mathtools}

\renewcommand{\Re}{\operatorname{Re}}

\DeclarePairedDelimiterX\innerp[2]{\langle}{\rangle}{#1, #2}
\DeclarePairedDelimiterX\ccint[2]{\lbrack}{\rbrack}{#1, #2}

\newtheorem{theorem}{Theorem}
\newtheorem{lemma}{Lemma}

\newtheorem{corollary}{Corollary}

\author{Petar Melentijevi\'c}
\thanks{The author is partially supported by MPNTR grant 174017, Serbia}
\title[Hypercontractive inequalities for weighted Bergman spaces]{Hypercontractive inequalities for weighted Bergman spaces}

\begin{document}
\begin{abstract} We obtain sharp $L^p\rightarrow L^q$ hypercontractive inequalities for the  weighted Bergman spaces on the unit disk $\mathbb{D}$ with the usual weights \\ $\frac{\alpha-1}{\pi}(1-|z|^2)^{\alpha-2},\alpha>1$ for $q\geq 2,$ thus solving an interesting case of a problem from \cite{JANSON}. We also give some estimates for $0<q<2.$ 
  
 \end{abstract}
\maketitle

\section{Introduction}

In his paper on hypercontractive inequalities for multipliers on orthogonal polynomials, (\cite{JANSON}), Janson posed a question on the hypercontractivity of the operator $P_r f(z):=f(rz)$ on spaces of analytic functions. Weissler (\cite{WEISSLER}) earlier proved the result for Hardy space, while Janson proved it for Fock spaces (using the estimates from \cite{NELSON}) and Bergman space with very specific weight. We will further investigate these estimates for the weighted Bergman spaces:
\begin{equation}
\bigg(\int_{\mathbb{D}}|f(rz)|^q dA_{\alpha}(z)\bigg)^{\frac{1}{q}}\leq \bigg(\int_{\mathbb{D}}|f(z)|^p dA_{\alpha}(z)\bigg)^{\frac{1}{p}},
\end{equation}
where $dA_{\alpha}(z)=\frac{\alpha-1}{\pi}(1-|z|^2)^{\alpha-2} dA(z)$ and $0<p<q<+\infty$ and $\alpha>1.$ Here $\mathbb{D}$ stands for the unit disk.
It can be easily seen, as in the setting of Hardy spaces, that $r$ cannot be greater that $\sqrt{\frac{p}{q}}.$ This value of $r$ comes from considering the family $f_{\epsilon}(z)=1+\epsilon z,$ for small $\epsilon'$s.

Weissler and Janson, in fact, proved (1) for $\alpha=1$ (Hardy space) and $\alpha=\frac{3}{2}.$ In \cite{BBHOCP}, using Beckner's estimates for Poisson semigroup on the unit sphere in $\mathbb{R}^n$ (\cite{BECKNER}), Bayart, Brevig, Haimi, Ortega-Cerda and Perfekt proved these inequalities for $\alpha=\frac{n+2}{2},$ with $n \in \mathbb{N}$ and any $0<p<q<\infty.$ In this note, we will prove those estimates for $q\geq 2$ and all $\alpha>1$ and give some partial results for $q<2.$ More precisely, we prove the following theorem:
\begin{theorem}
	Let $0<p<q<+\infty$ and $\alpha>1$. For $f \in A^p_{\alpha}$, i.e. an analytic function such that the right hand side of the inequality (1) is finite, we have:
	\begin{itemize}
	\item For $q\geq 2$ - the estimate (1) holds with $r=\sqrt{\frac{p}{q}}$
	\item For $q<2$ - (1) holds with $r^2=\max\{\frac{p}{2},\frac{p\alpha-p}{q\alpha-p}\}$ for an arbitrary holomorphic function, while for the zero-free function this estimate holds with $r=\sqrt{\frac{p}{q}}$
	\end{itemize}
\end{theorem}
By Janson's paper and very well-known argument on logarithmic Sobolev inequalities from \cite{GROSS}, we have the following:
\begin{corollary}
	There holds the logarithmic Sobolev inequality 
	$$  \int_{\mathbb{D}}|f(z)|^p\log|f(z)|dA_{\alpha}(z)\leq \frac{1}{2}\Re\int_{\mathbb{D}}|f(z)|^{p-2}\overline{f(z)}zf'(z) dA_{\alpha}(z)$$
	$$+   \frac{1}{p}\int_{\mathbb{D}}|f(z)|^pdA_{\alpha}(z)\log\bigg(\int_{\mathbb{D}}|f(z)|^pdA_{\alpha}(z)\bigg),$$
	for $p\geq 2$ and holomorphic function $f,$ where $dA_{\alpha}(z)=\frac{\alpha-1}{\pi}(1-|z|^2)^{\alpha-2}dA(z).$
\end{corollary}
Moreover, by \cite{GROSS}, this inequality is equivalent to (1).

\section{The method of a proof and the main result}

The main tools in our approach are recent inequalities due to Kulikov (\cite{KULIKOV}) and the lemma on convexity of certain integral means of analytic functions (\cite{PAVLOVIC}). Let us say that, in \cite{KULIKOV}, Kulikov has proved a conjecture of Pavlovi\'c from \cite{PAVLOVIC1} and \cite{PAVLOVIC} and Bayart, Brevig, Haimi, Ortega-Cerda and Perfekt from \cite{BBHOCP}, where several results that support this and some other conjectures are provided. Moreover, he proved a more general conjecture of Lieb and Solovej from \cite{LIEBSOLOVEJ}. Some interesting partial results are given in \cite{DANGUOWANG}, \cite{KULIKOV1} and \cite{LLINEARES}. See also \cite{KALAJ} and \cite{FRANK} for some generalizations and extensions of these results. Now, we will formulate Kulikov's theorem from \cite{KULIKOV}(in the appropriate form) and then the lemma from \cite{PAVLOVIC}, including its proof for the sake of completeness. 

\begin{theorem}
	For $0<p<q<+\infty$ and $f \in A^p_{\alpha}(\mathbb{D})$ there holds the inequality
	\begin{equation}
	\bigg(\int_{\mathbb{D}}|f(z)|^q dA_{\beta}(z)\bigg)^{\frac{1}{q}}\leq  \bigg(\int_{\mathbb{D}}|f(z)|^p dA_{\alpha}(z)\bigg)^{\frac{1}{p}},
	\end{equation} 
	where $dA_{\alpha}$ is defined as before and $\frac{\alpha}{p}=\frac{\beta}{q}.$ The inequality is sharp and the extremizers for the fixed $p,q, \alpha, \beta$ are the functions given by $f(z)=C(1-az)^{\frac{-2\alpha}{p}}$ with $C \in \mathbb{C}, a\in \mathbb{D}.$
\end{theorem}

\begin{lemma}
	For $f$ analytic in the unit disk $\mathbb{D}$ and $p\geq 2,$ the function
	$$ \Phi_p(r)=\frac{1}{2\pi}\int_{0}^{2\pi}|f(\sqrt{r}e^{\imath \theta})|^p d\theta$$
	is convex on $r.$ The same function is convex on $r$ for $0<p<2$ for zero-free function $f.$
\end{lemma}
\begin{proof}
We start from the formula
$$\Phi_p(r)=|f(0)|^p+\frac{p^2}{2\pi}\int_{0}^{\sqrt{r}}\frac{1}{\rho}\bigg(\int_{|z|<\rho}|f(z)|^{p-2}|f'(z)|^2 dA(z)\bigg)d\rho,$$
from whom we easily find
$$\Phi'_p(r)=\frac{p^2}{4\pi r}\int_{|z|<\sqrt{r}}|f(z)|^{p-2}|f'(z)|^2 dA(z) $$
$$=\frac{p^2}{4\pi r}\int_{0}^{\sqrt{r}}\rho\bigg(\int_{0}^{2\pi}|f(\rho e^{\imath\theta})|^{p-2}|f'(\rho e^{\imath\theta})|^2d\theta\bigg)d\rho$$
and
$$\Phi''_p(r)=\frac{p^2}{4\pi r^2}\bigg(\frac{r}{2}\int_{0}^{2\pi}|f(\sqrt{r} e^{\imath\theta})|^{p-2}|f'(\sqrt{r} e^{\imath\theta})|^2d\theta-$$
$$\int_{0}^{\sqrt{r}}\rho\big(\int_{0}^{2\pi}|f(\rho e^{\imath\theta})|^{p-2}|f'(\rho e^{\imath\theta})|^2d\theta\big)d\rho\bigg).$$
If we denote $F(\rho)=\int_{0}^{2\pi}|f(\rho e^{\imath\theta})|^{p-2}|f'(\rho e^{\imath\theta})|^2d\theta,$ then from its subharmonicity we would have
$$\int_{0}^{\sqrt{r}}\rho F(\rho)d\rho\leq\int_{0}^{\sqrt{r}}\rho F(\sqrt{r})d\rho=\frac{r}{2}F(\sqrt{r})$$
and, by the last formula, we can conclude $\Phi_p''(r)\geq 0.$
However, we easily find
$$\Delta(|f(z)|^{p-2}|f'(z)|^2)$$
$$=|f(z)|^{p-4}\bigg((p-2)^2|f'(z)|^4+4|f(z)|^2|f''(z)|^2-4(p-2)\Re\big(f(z)f''(z)\overline{f'(z)}^2\big)\bigg)$$
$$=|f(z)|^{p-4}|(p-2)f'(z)^2-2f(z)f''(z)|^2\geq 0$$
and the conclusion follows.
\end{proof}
Now we prove our Theorem 1. 
\begin{proof}
First, we estimate the right-hand side from below using Kulikov's inequality (2): 
$$\bigg(\frac{\beta-1}{\pi}\int_{\mathbb{D}}|f(z)|^q(1-|z|^2)^{\beta-2}dA(z)\bigg)^{\frac{1}{q}}\leq \bigg(\frac{\alpha-1}{\pi}\int_{\mathbb{D}}|f(z)|^p(1-|z|^2)^{\alpha-2}dA(z)\bigg)^{\frac{1}{p}}.$$
Our estimate will follow from the inequality: 
\begin{equation}
	\frac{\alpha-1}{\pi}\int_{\mathbb{D}}|f(rz)|^q(1-|z|^2)^{\alpha-2}dA(z)\leq \frac{\beta-1}{\pi}\int_{\mathbb{D}}|f(z)|^q(1-|z|^2)^{\beta-2}dA(z).
\end{equation}
Using polar coordinates, we get that the last inequality is equivalent to 
$$(\alpha-1)\int_{0}^{1}\rho(1-\rho^2)^{\alpha-2}\bigg(\int_{0}^{2\pi}|f(r\rho e^{\imath\theta})|^q d\theta\bigg)d\rho$$
$$\leq (\beta-1)\int_{0}^{1}\rho(1-\rho^2)^{\beta-2}\bigg(\int_{0}^{2\pi}|f(\rho e^{\imath\theta})|^q d\theta\bigg)d\rho.$$
After the change of variable $\rho=\sqrt{y}$ and introducing the function $\Phi(y)=\int_{0}^{2\pi}|f(\sqrt{y}e^{\imath \theta})|^q d\theta,$ we get:
$$(\alpha-1)\int_{0}^{1}(1-y)^{\alpha-2}\Phi(r^2y)dy\leq (\beta-1)\int_{0}^{1}(1-y)^{\beta-2}\Phi(y)dy.$$
Integrating by parts gives
\begin{equation}
r^2\int_{0}^{1}(1-y)^{\alpha-1}\Phi'(r^2y)dy\leq \int_{0}^{1}(1-y)^{\beta-1}\Phi'(y)dy.
\end{equation}
Integrating by parts once again, we get
$$\frac{r^2}{\alpha}\Phi'(0)+\frac{r^4}{\alpha}\int_{0}^{1}(1-y)^{\alpha}\Phi''(r^2y)dy\leq \frac{1}{\beta}\Phi'(0)+\frac{1}{\beta}\int_{0}^{1}(1-y)^{\beta}\Phi''(y)dy,$$
or, by using that $r^2=\frac{p}{q}=\frac{\alpha}{\beta}:$
$$r^2\int_{0}^{1}(1-y)^{\alpha}\Phi''(r^2y)dy\leq \int_{0}^{1}(1-y)^{\beta}\Phi''(y)dy.$$
Since 
 $$r^2\int_{0}^{1}(1-y)^{\alpha}\Phi''(r^2y)dy=\int_{0}^{r^2}(1-\frac{y}{r^2})^{\alpha}\Phi''(y)dy$$
and Lemma 1 gives the convexity of $\Phi(y),$ i.e. $\Phi''(y)\geq 0$ the result follows 
from the inequality 
$$\int_{0}^{r^2}(1-\frac{y}{r^2})^{\alpha}\Phi''(y)dy\leq  \int_{0}^{r^2}(1-y)^{\beta}\Phi''(y)dy,$$
which is implied by the inequality
\begin{equation}
(1-\frac{y}{r^2})^{\alpha}\leq (1-y)^{\beta}\quad \text{for}\quad 0\leq y\leq \frac{\alpha}{\beta}=r^2.
\end{equation}
But, $g(y)=(1-y)^{\frac{\beta}{\alpha}}$ is convex and therefore not smaller than $g(0)+g'(0)y=1-\frac{y}{r^2}.$ This finishes the proof for the case $q\geq 2.$

For an arbitrary holomorphic function $f$, the function $\Phi(y)$ need not be convex, but it is still monotone increasing (since $|f(z)|^q$ is subharmonic for holomorphic $f$ and $q>0$)
and we can prove (4) with $r^2=\frac{\alpha-1}{\beta-1}$ using the inequality (5) with $\alpha-1$ and $\beta-1$ instead of $\alpha$ and $\beta,$ respectively. The second estimate for $r$ in the case $q<2$ follows from the first part of Theorem with $q=2$ and Jensen's inequality $\|f(\sqrt{\frac{p}{2}}z)\|_{A^q_{\alpha}}\leq \|f(\sqrt{\frac{p}{2}}z)\|_{A^2_{\alpha}}.$

For the zero-free function $f$ and $0<p<q$ with $q<2$ we can use the first part of this Theorem for the function $f^{\frac{1}{n}}$ and $np$ and $nq$ instead of $p$ and $q,$ respectively, where $n$ is such that $nq\geq 2.$   
\end{proof}

\textbf{Remark}. The careful reader can easily deduce that the statement of our theorem also holds for functions with the zeros whose multiplicities are all big enough depending on $p.$ For example, if a function $f$ has all zeros in $\mathbb{D}$ of multiplicities $2,$ than its square root function is well defined and the desired estimate (1) holds for $q\geq 1.$ Also, it can be easily verified that the estimate (3) is not valid for $f(z)=z$ and $q<2.$ However, using log-Sobolev inequality from Corollary 1 one can see from some estimates of Gamma function and its derivatives that (1) holds for $f(z)=z$ for all $\alpha>1$ and $0<p<q<+\infty.$ 

    After submitting this paper, we are informed that Adrian Llineares and Alejandro Mas have proved the similar result using different techniques.

\textbf{Acknowledgement}. The author is grateful to David Kalaj and Dragan Vukoti\'c for their encouragement and inspiring conversations on this topic. The author is supported by the MPNTR grant 174017.

\end{document}